\documentclass[12pt,a4paper]{amsart}  
\usepackage{amsthm}
\usepackage{graphicx} 
\usepackage{color}

\usepackage{calrsfs}

\newcommand{\R}{\mathbb{R}}
\newcommand{\C}{\mathbb{C}}

\newcommand {\Id}{{\mathbb I}}
\newcommand{\DD}{\mathcal{D}}
\renewcommand{\SS}{\mathcal{S}}

\newcommand{\Div}{\mathop{\mathrm{div}}}
\newcommand{\curl}{\mathop{\mathrm{curl}}}
\newcommand{\vcurl}{\mathop{\mathbf{curl}}}

\newcommand{\grad}{\mathop{\boldsymbol{\nabla}}}
\newcommand{\im}{\mathop{\mathrm{im}}}
\newcommand{\codim}{\mathop{\mathrm{codim}}}

\newcommand{\bE}{\mathbf{E}}
\newcommand{\bF}{\mathbf{F}}
\newcommand{\bff}{\mathbf{f}}
\newcommand{\bH}{\mathbf{H}}

\newcommand{\bj}{\mathbf{j}}
\newcommand{\bJ}{\mathbf{J}}

\newcommand{\bn}{\mathbf{n}}
\newcommand{\bt}{\mathbf{t}}

\newcommand{\bX}{\mathbf{X}}

\newcommand{\Sess}{\sigma_{\mathrm{ess}}}                                                                                                 

\newtheorem{theorem}{Theorem}[section]
\newtheorem{lemma}[theorem]{Lemma}
\newtheorem{proposition}[theorem]{Proposition}
\newtheorem{corollary}[theorem]{Corollary}

\theoremstyle{definition}

\theoremstyle{remark}

\numberwithin{equation}{section}

\begin{document}


\title[Volume integral equations for 2D electromagnetism]{Volume integral equations for electromagnetic scattering in two dimensions}

\author{Martin Costabel, Eric Darrigrand and Hamdi Sakly}

\address{IRMAR UMR 6625 du CNRS, Universit\'{e} de Rennes 1 }

\address{Campus de Beaulieu,
35042 Rennes Cedex, France \bigskip}

\begin{abstract}
We study the strongly singular volume integral equation that describes the scattering of time-harmonic electromagnetic waves by a penetrable obstacle. We consider the case of a cylindrical obstacle and fields invariant along the axis of the cylinder, which allows the reduction to two-dimensional problems. With this simplification, we can refine the analysis of the essential spectrum of the volume integral operator started in a previous paper
({\sc M.~Costabel, E.~Darrigrand, and H.~Sakly}, {\em The essential spectrum of
  the volume integral operator in electromagnetic scattering by a homogeneous
  body}, Comptes Rendus Mathematique, 350 (2012), pp.~193--197)
and obtain results for non-smooth domains that were previously available only for smooth domains. It turns out that in the TE case, the magnetic contrast has no influence on the Fredholm properties of the problem. As a byproduct of the choice that exists between a vectorial and a scalar volume integral equation, we discover new results about the symmetry of the spectrum of the double layer boundary integral operator on Lipschitz domains.
\end{abstract}


\maketitle

\section{Introduction}\label{S:Intro}

The scattering of electromagnetic waves by a penetrable object for a fixed frequency is described by the time-harmonic Maxwell equations, valid in the whole space, even if the scatterer is confined to a bounded region. For purposes of numerical modeling, one prefers a problem posed on a bounded domain, and various methods are known to achieve this. One way is the method of volume integral equations (VIEs) that will be described in Section~\ref{S:TE-VIE}. This method has been used for a long time in computational physics \cite{Jarvenpaa2013,Yurkin20112234}, but its mathematical analysis is still incomplete. The case of coefficients, permittivity $\epsilon$ and permeability $\mu$, that are smooth functions on the whole space has been studied more completely, because the VIE can then be reduced to standard Fredholm integral equations of the second kind, see \cite{Botha2006,BudkoSamokhin2006,Rahola2000} and the chapter on Lippmann-Schwinger equations in the classical monograph \cite{ColtonKress1998}. 

We are rather interested in the case where the coefficients are discontinuous on the boundary of the scatterer, as is the case at the interface between different materials. The well-posedness of the strongly singular VIEs is then a non-trivial problem. There exist negative or complex-valued coefficients $\epsilon$ and $\mu$ that can nowadays be realized with meta-materials, for which the scattering problem and its associated VIE are not well posed in the Fredholm sense. For the scattering problem, there has been recent progress in finding necessary and sufficient conditions for Fredholmness using variational techniques, see \cite{BBChesnelCiarlet_2D2014} for the two-dimensional situation. For the VIEs, the question amounts to determining the essential spectrum of the volume integral operators.  
The knowledge of the essential spectrum is also of basic importance for the analysis of numerical algorithms, for example for the construction of efficient iterative solution procedures, see \cite{SamokhinShestopalovKobayashi2013}. 

In \cite{CoDarSak2012}, we analyzed the VIE using a technique of extension to a coupled boundary-domain integral equation system, which eventually allows the reduction of the problem to the better known analysis of boundary integral operators. In the three-dimensional case studied in \cite{CoDarSak2012}, the analysis of the ``magnetic'' volume integral operator poses particular difficulties, which is the reason why we could describe its essential spectrum only under smoothness assumptions on the boundary and why the analysis given in \cite{Kirsch2007,KirschLechleiter2009} is incomplete for this case.

In this paper we study the two-dimensional case, where it turns out that this ``magnetic'' operator $B_{k}^{\nu}$ has a very simple structure, see subsection~\ref{SSS:mag} below, in particular Proposition~\ref{P:B0projector}. This structure implies that the operator, while still strongly singular and therefore contributing to the principal part of the VIE, does not contribute in a non-trivial way to the essential spectrum. It also suggests a new idea for the analysis of this operator in the three-dimensional case that works also when the boundary is only Lipschitz, see \cite[Section~2.4]{sakly2014} and our forthcoming paper on the subject. Thus, while the 2D case has its own interest, for example for the study of waveguides or photonic fibres, our main motivation for studying it here is that it allows a refined spectral analysis of the volume integral operators. 

Another interesting feature of the two-dimensional electromagnetic problem is that one can choose to describe it either as a vector-valued or as a scalar problem. This leads to two different VIEs. In both cases, we prove necessary and sufficient conditions on the coefficients $\epsilon$ and $\mu$ for the volume integral operators to be Fredholm. These conditions involve the essential spectrum of certain well-studied scalar boundary integral operators, and at first it looks like the two formulations lead to different conditions. They are equivalent, however, because of a symmetry property of the essential spectrum of the double layer boundary integral operator that is specific to dimension two. We prove this symmetry in Section~\ref{S:sym} by relating the boundary integral equation to a scalar transmission problem involving the coefficient $\epsilon$ and  for which one finds a symmetry with respect to replacing $\epsilon$ by $\frac1\epsilon$.

\section{The two dimensional Maxwell equations}\label{S:2DMax}

We consider an electromagnetic scattering problem described by the time-harmonic Maxwell equations in $\R^{3}$
\begin{equation}
\label{E:Max}
   \curl E - i\omega\mu H = 0\,;\qquad  
   \curl H + i\omega\varepsilon E = J\,.
\end{equation}
The current density $J$ (which may be identically zero) is related to an incident field $(E^{0},H^{0})$ by the requirement that the latter satisfies the free-space Maxwell equations
\begin{equation}
\label{E:inc}
   \curl E^{0} - i\omega\mu_{0} H^{0} = 0 \,;\qquad  
   \curl H^{0} + i\omega\varepsilon_{0} E^{0} = J 
\end{equation} 
with constant permittivity $\varepsilon_{0}$ and permeability $\mu_{0}$. 
In the Maxwell equations \eqref{E:Max}, the fields are supposed to be locally square integrable, and the equations are satisfied in the distributional sense. This means that on a surface of discontinuity of the coefficients $\mu$ and $\varepsilon$, the fields satisfy transmission conditions
\begin{equation}
\label{E:Trans3D}
 [n\times E]=0;\quad
 [n\times H]=0;\quad
 [n\cdot\mu H]=0;\quad
 [n\cdot\varepsilon E]=0\,.
\end{equation} 
The scattering problem is completed by the requirement that $E-E^{0}$ and $H-H^{0}$ satisfy a radiation condition at infinity.
Note that the incident field may contain a part generated by the current density $J$ and satisfying the radiation condition, as well as a part satisfying the homogeneous Maxwell system but not the radiation condition. 

We assume now that the penetrable obstacle is of cylindrical form $\Omega\times\R$, where $\Omega$ is a bounded Lipschitz domain in $\R^{2}$. 
In addition, we assume that the fields $E$ and $H$ as well as the incident fields do not depend on the variable $x_{3}$. 
We introduce the relative permittivity and permeability such that 
$\varepsilon=\varepsilon_{r}\varepsilon_{0}$, $\mu=\mu_{r}\mu_{0}$. 
For simplicity we further assume that $J$ has compact support disjoint from the scatterer $\Omega$.

Introducing the wave number $k$ such that $k^{2}=\omega^{2}\varepsilon_{0}\mu_{0}$ and suitably renormalizing the fields,
the resulting system of partial differential equations in $\R^{2}$ can be written as
\begin{equation}\label{E:2D_Max}
\begin{aligned}
\partial_{2}E_3-i\,kH_1&=m_1\;;\;\\
-\partial_1E_3-i\,kH_2&=m_2\;;\;\\
\partial_1E_2-\partial_2E_1-i\,kH_3&=m_3\;;\;
\end{aligned}
\begin{aligned}
\partial_{2}H_3+i\,kE_1&=j_1\;;\\
-\partial_1H_3+i\,kE_2&=j_2\;;\\
\partial_1H_2-\partial_2H_1+i\,kE_3&=j_3\;.
\end{aligned}
\end{equation}
where 
\begin{equation}
\label{E:rhs}
\mbox{
$(m_{1},m_{2},m_{3})^{\top}=ik(\mu_{r}-1)\chi_{\Omega}H$,\quad
$(j_{1},j_{2},j_{3})^{\top}=ik(1-\varepsilon_{r})\chi_{\Omega} E + J$
}
\end{equation}
and $\chi_{\Omega}$ denotes the characteristic function of $\Omega$.

It is well known (and easy to see from \eqref{E:2D_Max}) that this system can be written as two uncoupled systems corresponding to the transverse electric (TE) and transverse magnetic (TM) polarizations. 

In the TE case, only the transverse components $E_{1}, E_{2}$ of the electric field and the longitudinal component $H_{3}$ of the magnetic field appear:
\begin{equation}\label{E:TE}
\begin{aligned}
\partial_1E_2-\partial_2E_1-i\,kH_3&=m_3\;;\;\\
\partial_{2}H_3+i\,kE_1&=j_1\;;\\
-\partial_1H_3+i\,kE_2&=j_2\;.
\end{aligned}
\end{equation}
Note that we understand \eqref{E:TE} in the distributional sense, which implies that the correct transmission conditions for $H_{3}$ and the 2D tangential component of $E$  are satisfied on $\Gamma=\partial\Omega$.

The TM case is described by the remaining 3 equations of the system \eqref{E:2D_Max}:
\begin{equation}\label{E:TM}
\begin{aligned}
\partial_{2}E_3-i\,kH_1&=m_1\;;\;\\
-\partial_1E_3-i\,kH_2&=m_2\;;\;\\
\partial_1H_2-\partial_2H_1+i\,kE_3&=j_3\;.
\end{aligned}
\end{equation}

We observe the symmetry between \eqref{E:TE} and \eqref{E:TM}, which reflects the symmetry between $E$ and $H$ in the original Maxwell system. It is therefore sufficient to discuss one of the two cases in detail, and we choose the TE system \eqref{E:TE}. 

\section{Integral equation formulations for the (TE) case }\label{S:TE-VIE}

There are two ways of reducing the first-order system \eqref{E:TE} to second order by eliminating one of the two fields, and these two second-order formulations are no longer identical, one of them leading to a second order system that can be called 2D Maxwell system, and the other one to a scalar Helmholtz equation.
Correspondingly, we can construct two different volume integral equations, both of them equivalent to the TE system \eqref{E:TE}, but bringing forward different aspects of the volume integral operators. 

In the ``electric'' formulation we have a $2\times2$ system of second order partial differential equations, and correspondingly we find a vector-valued VIE. The matrix of strongly singular integral operators is the sum of an ``electric'' operator $A_{k}^{\eta}$ depending on the permittivity $\epsilon_{r}$ and a ``magnetic'' operator $B_{k}^{\nu}$ depending on the permeability $\mu_{r}$. In the three-dimensional case studied in \cite{CoDarSak2012}, the electric volume integral operator was easier to analyse, using a method of extension to a coupled boundary-domain integral equations already described in \cite{CoDarKo2010}. A much earlier analysis (where this operator is called ``magnetic'') can be found in \cite{FriedmanPasciak1984}.
The magnetic volume integral operator $B_{k}^{\nu}$ is more difficult to analyze in three dimensions, but it turns out that here in two dimensions, the operator $B_{k}^{\nu}$ is a compact perturbation of a projection operator, see Proposition~\ref{P:B0projector}. This simple structure implies that the magnetic contrast does not influence the essential spectrum, see Corollary~\ref{C:A+BFredholm}.

In the ``magnetic'' formulation, the volume integral operator is a scalar strongly singular integral operator, similar to what one gets when acoustic scattering problems are treated with the VIE method \cite{Costabel:2014arXiv1412.8685C}. The operator again splits into an electric part and a magnetic part, but this time the magnetic operator is weakly singular and therefore disappears altogether from the analysis of the essential spectrum.

For the derivation of the volume integral equation (VIE), we assume in this section that the wave number $k$ is a positive real number. Later on, seeing that the operators depend analytically on the parameter $k$, we may consider arbitrary complex values for $k$.

The TE problem consists of the system of three first order partial differential equations \eqref{E:TE} with compactly supported right hand side, completed by the Silver-M\"uller radiation condition, which is equivalent to the Sommerfeld radiation condition
\begin{equation}
\label{E:Sommerfeld}
\lim_{R\to\infty}\int_{|x|=R}|\partial_{r}u-iku|^{2}ds=0
\end{equation}
for each of the three components $u=E_{1}-E_{1}^{0},E_{2}-E_{2}^{0},H_{3}-H_{3}^{0}$ of the scattered field. Note that the incident field itself may satisfy the radiation condition if it is generated by a current density $J$ that does not vanish identically, but our formulation can also handle the case where $J$ is zero and $E^{0}$ does not satisfy the radiation condition, as is the case e.g.\ for an incident plane wave.

In our scattering problem, the right hand side contains the imposed current density and also the induced polarization, which is given by the unknown fields multiplied by the inhomogeneity in the coefficients, see equations \eqref{E:rhs}.

\subsection{The electric formulation: $E_1,E_2$ as unknowns}\label{SS:E12}

In the following, we use the two standard instantiations of the curl operator in two dimensions:
$$
 \curl \bE = \partial_{1}E_{2} - \partial_{2}E_{1}, \qquad
 \vcurl u = (\partial_{2}u, -\partial_{1}u)^{\top}
$$
and write \eqref{E:TE} as
\begin{equation}
\label{E:TEcurl}
 \curl\bE -ik H_{3} = m_{3}\,;\qquad \vcurl H_{3} + ik\bE = \bj\,.
\end{equation}
Eliminating $H_{3}$ from this system gives a second order system for the unknown $\bE=(E_{1},E_{2})^{\top}$
\begin{equation}
\label{E:PDE2-E12}
 \vcurl\curl\bE -k^{2}\bE = \vcurl m_{3}+ik\bj\,.
\end{equation}
The construction of volume integral equations is based on the following well known result.

\begin{lemma}
\label{L:repres}
Let $g_{k}(x) = \frac{i}{4}H_0^{(1)}(k|x|)$ be the fundamental solution of the Helmholtz equation in $\R^{2}$ satisfying the outgoing Sommerfeld radiation condition.\\
\mbox{\emph{(i)}}
The distribution $u$ satisfies in $\R^{2}$ the Helmholtz equation
$$
  -(\Delta+k^{2})u = f
$$
and the radiation condition \eqref{E:Sommerfeld}, where $f$ is a distribution with compact support, if and only if 
$$
  u = g_{k}* f\,.
$$
If $f$ is an integrable function, then this convolution can be written as an integral:
$$
  u(x) = \int g_{k}(x-y)\,f(y)\,dy \;.
$$
\mbox{\emph{(ii)}}
The vector-valued distribution $\bE$ satisfies in $\R^{2}$ the system
$$
  \vcurl\curl\bE -k^{2}\bE = \bff
$$
and the radiation condition, where $\bff$ has compact support, if and only if
$$
  \bE = k^{-2}(\grad\Div+k^{2}) g_{k}* \bff \,.
$$
\end{lemma}
Note that the vector part (ii) of this lemma follows from the scalar part (i) and the relation 
$$
  (\grad\Div+k^{2}) (\vcurl\curl - k^{2}) = -k^{2} (\Delta + k^{2})
$$
if one uses the fact that for a function $u$ satisfying the radiation condition also the derivatives of $u$ satisfy the radiation condition and that derivations commute with convolutions. 

Applying Lemma~\ref{L:repres} to the system \eqref{E:PDE2-E12} and using the system \eqref{E:inc} satisfied by the incoming field, we obtain
$$
  \bE-\bE^{0} = \vcurl g_{k}*m_{3} - \frac1{ik}\big(\grad\Div+k^{2}\big) g_{k}*\bj 
$$
where the right hand side is given by \eqref{E:rhs}, but with $J=0$. Eliminating $H_{3}$ from the right hand side using the equation
$$
  \curl\bE -ik\mu_{r}H_{3}=0\,,
$$
we finally obtain
$$
  \bE - \bE^{0}= \vcurl g_{k}* \big(\frac{\mu_{r}-1}{\mu_{r}}\chi_{\Omega}\curl\bE\big)
        - (\grad\Div+k^{2})g_{k}*((1-\epsilon_{r})\chi_{\Omega}\bE)\,.
$$
Thus our VIE takes the form
\begin{equation}
\label{E:VIE-TE-E12}
 \bE - A_{k}^{\eta}\bE - B_{k}^{\nu}\bE = \bE^{0}
\end{equation}
where we used the abbreviations
\begin{equation}
\label{E:eta-nu}
\eta=1-\epsilon_{r}\,,\qquad \nu=1-\tfrac1{\mu_{r}}
\end{equation}
\begin{align}
\label{E:A}
 A_{k}^{\eta}\bE &= -(\grad\Div+k^{2})g_{k}*(\eta\chi_{\Omega}\bE)\\
\label{E:B}
 B_{k}^{\nu}\bE &= \vcurl g_{k}* (\nu\chi_{\Omega}\curl\bE)\,.
\end{align}
Strictly speaking, we would not need to write the characteristic function $\chi_{\Omega}$ here, because the coefficient functions $\eta$ and $\nu$ describing the electric and magnetic contrast vanish outside of $\Omega$, but this notation allows us to treat $\eta$ and $\nu$ simply as numbers in the piecewise constant case.

The VIE \eqref{E:VIE-TE-E12} is valid on all of $\R^{2}$. However, as soon as $\bE$ is known on $\Omega$, the values of $A_{k}^{\eta}\bE$ and $B_{k}^{\nu}\bE$ are known, and then \eqref{E:VIE-TE-E12} defines $\bE$ on the entire space $\R^{2}$. Thus we will consider the VIE \eqref{E:VIE-TE-E12} as an integral equation on $\Omega$ for the unknown $\bE$  on $\Omega$, and then use the formula \eqref{E:VIE-TE-E12} to extend $\bE$ to all of $\R^{2}$.

Both integral operators $A_{k}^{\eta}$ and $B_{k}^{\nu}$ are integro-differential operators that can be considered as integral operators with strongly singular kernels, since the second derivatives of $g_{k}(x-y)$ behave like $|x-y|^{-2}$ as $x-y\to0$.

We use the standard notation for Sobolev spaces $H^{m}$ and for 
\begin{align*}
\bH(\curl,\Omega) &= \{\bE\in L^{2}(\Omega)^{2} \mid \curl \bE \in L^{2}(\Omega) \}\\
\bH(\Div,\Omega) &= \{\bE\in L^{2}(\Omega)^{2} \mid \Div \bE \in L^{2}(\Omega) \}\,.
\end{align*}

\begin{lemma}
\label{L:ABinHcurl}
Let $\eta,\nu\in L^{\infty}(\Omega)$. Then
the operators $A_{k}^{\eta}$ and $B_{k}^{\nu}$ are bounded linear operators in $\bH(\curl,\Omega)$.
\end{lemma}
\begin{proof}
The convolution with $g_{k}$ is a pseudodifferential operator of order $-2$ and therefore maps any $L^{2}$ function with compact support to a function in the Sobolev space 
$H^{2}_{\rm loc}(\R^{2})$. For a bounded Lipschitz domain $\Omega$, the restriction of this convolution to $\Omega$ therefore defines a bounded operator from $L^{2}(\Omega)$ to $H^{2}(\Omega)$. 

This implies that $A_{k}^{\eta}$ is bounded from $L^{2}(\Omega)^{2}$ to $L^{2}(\Omega)^{2}$ and $\curl A_{k}^{\eta}$ even maps $L^{2}(\Omega)^{2}$ boundedly into $H^{1}(\Omega)$. Hence $A_{k}^{\eta}$ is bounded from $L^{2}(\Omega)^{2}$ to $\bH(\curl,\Omega)$.

For $B_{k}^{\nu}$, we see that it maps $\bH(\curl,\Omega)$ boundedly to $H^{1}(\Omega)^{2}$, hence to $\bH(\curl,\Omega)$. 
\end{proof}
The main question to be answered here is under what conditions on the coefficients $\eta,\nu$ the volume integral operator
$\Id-A_{k}^{\eta}-B_{k}^{\nu}$ is Fredholm in the space $\bH(\curl,\Omega)$.
Equivalently, when is the number $1$ not a member of the essential spectrum (see \eqref{E:Sess}) of the operator $A_{k}^{\eta}+B_{k}^{\nu}$. More generally, the aim is to determine 
$\Sess(A_{k}^{\eta}+B_{k}^{\nu})$ in dependence of the coefficients $\eta$ and $\nu$.
This is the subject of Section~\ref{S:spec} below.

\subsection{The magnetic formulation: $H_{3}$ as unknown}\label{SS:H3}

We start once more from the system \eqref{E:TEcurl} and reduce it to a second order equation, this time by eliminating the electric field:
\begin{equation}
\label{E:PDE2-H3}
 -(\Delta+k^{2})H_{3}= \curl\bj -ikm_{3}\,.
\end{equation}
The application of Lemma~\ref{L:repres} leads to
$$
 H_{3}-H_{3}^{0} = g_{k}*\curl\bj -ik g_{k}*m_{3}
   = ik\curl g_{k}*(\eta\chi_{\Omega}\bE) 
     + k^{2}g_{k}*\big((\mu_{r}-1)\chi_{\Omega}H_{3}\big)\,.
$$
The remaining $\bE$ on the right hand side is eliminated using 
$$
 \vcurl H_{3} + ik\epsilon_{r}\bE = \bJ\,, 
$$
finally leading to
$$
  H_{3}-H_{3}^{0} = \curl g_{k}*(\alpha\chi_{\Omega}\vcurl H_{3})
     + k^{2}g_{k}*(\beta\chi_{\Omega}H_{3})
$$
with 
\begin{equation}
\label{E:defab}
 \alpha=1-\tfrac1{\epsilon_{r}}, \qquad  \beta=\mu_{r}-1\,.
\end{equation}
Here we used the simplifying assumption that the current density is supported in the complement of $\Omega$, so that $\chi_{\Omega}\bJ=0$. Otherwise there would appear another term on the right hand side.

We can write this VIE in the form
\begin{equation}
\label{E:VIE-TE-H3}
 H_{3} - C_{k}^{\alpha}H_{3} - D_{k}^{\beta}H_{3} = H_{3}^{0}
\end{equation}
with the integral operators
\begin{align}
\label{E:C}
 C_{k}^{\alpha}u &= \curl g_{k}*(\alpha\chi_{\Omega}\vcurl u)\\
\label{E:D}
 D_{k}^{\beta}H_{3} &= k^{2}g_{k}*(\beta\chi_{\Omega} u)\,.
\end{align}
The natural function space is now simply $H^{1}(\Omega)$, and the mapping properties of these two integral operators are easy to observe.
\begin{lemma}
\label{L:CDinH1}
Let $\alpha,\beta \in L^{\infty}(\Omega)$. Then
$C_{k}^{\alpha}$ is a bounded operator in $H^{1}(\Omega)$.
The operator
$D_{k}^{\beta}$ maps $L^{2}(\Omega)$ boundedly to $H^{2}(\Omega)$, thus it is compact in $H^{1}(\Omega)$.
\end{lemma}
The question to be answered in the next section is to determine 
$\sigma_{\rm ess}(C_{k}^{\alpha}+D_{k}^{\beta})$.

\section{Analysis of the VIE in the TE case}\label{S:spec}

In order to determine the Fredholm properties of the VIEs \eqref{E:VIE-TE-E12} and \eqref{E:VIE-TE-H3}, we study the essential spectrum of the operators $A_{k}^{\eta}+B_{k}^{\nu}$ and $C_{k}^{\alpha}+D_{k}^{\beta}$ using methods developed in \cite{CoDarKo2010,CoDarSak2012,Costabel:2014arXiv1412.8685C} and some properties specific to the two-dimensional case.

As definition of the essential spectrum $\Sess(A)$ of a bounded operator in a Banach space we choose
\begin{equation}
\label{E:Sess}
  \Sess(A) = \{\lambda\in\C \mid \lambda\Id - A 
               \mbox{ is not Fredholm} \}\,.
\end{equation}
Among the many possible different definitions of the essential spectrum, this one is sometimes called the Wolf essential spectrum. It is invariant under compact perturbations. Many of the operators studied in the following, in particular in the case of piecewise constant coefficients, can be reduced to compact perturbations of selfadjoint operators in some Hilbert spaces, which implies that if they are Fredholm, they are of index zero.  

In this section, we first assume that the coefficients are piecewise constant. This means that $\alpha, \beta, \eta, \nu$ are complex numbers and the volume integral operators depend linearly on these numbers:
$$
 A_{k}^{\eta} = \eta A_{k}^{1}\,;\quad
 B_{k}^{\nu}  = \nu B_{k}^{1}\,;\quad
 C_{k}^{\alpha} = \alpha C_{k}^{1}\,;\quad
 D_{k}^{\beta} = \beta D_{k}^{1}\,.
$$

 The essential spectrum being invariant with respect to compact perturbations of the operators, we can further simplify the VIEs by replacing the fundamental solution of the Helmholtz equation by the one of the Laplace equation.
\begin{lemma}
\label{L:gk-g0}
Let $g_{0}(x)=-\frac1{2\pi}\log|x|$.
Then for any $k\in\C$ the convolution by $g_{k}-g_{0}$ is a pseudodifferential operator of order $-4$ on $\R^{2}$.
\end{lemma} 
\begin{proof}
This can be found in the book \cite{NedelecLivre}, but it is seen most easily by comparing the symbols, i.e.\ the Fourier transforms:
$$
  \frac{1}{|\xi|^{2}-k^{2}} - \frac{1}{|\xi|^{2}} = \frac{k^{2}}{(|\xi|^{2}-k^{2})|\xi|^{2}}
$$
\end{proof}
In particular, the convolution by this difference  $g_{k}-g_{0}$ maps $L^{2}(\Omega)$ boundedly to $H^{4}(\Omega)$ for any bounded Lipschitz domain $\Omega$.

We now let the operators $A_{0}$, $B_{0}$, $C_{0}$, $D_{0}$ be defined like $A^{1}_{k}$, $B^{1}_{k}$, $C^{1}_{k}$, $D^{1}_{k}$, but for $k=0$. More precisely
\begin{align}
\label{E:A0}
 A_{0}\bE &= -(\grad\Div)g_{0}*(\chi_{\Omega}\bE)\\
\label{E:B0}
 B_{0}\bE &= \;\vcurl g_{0}* (\chi_{\Omega}\curl\bE)\\
\label{E:C0}
 C_{0} u &= \;\,\curl g_{0}*(\chi_{\Omega}\vcurl u)\\
\label{E:D0}
 D_{0} u &= \;0\,. 
\end{align}

\begin{corollary}
\label{C:compperturb}
The operators
$$\begin{aligned}
 A^{1}_{k}-A_{0}, B^{1}_{k}-B_{0}&: \bH(\curl,\Omega)\to \bH(\curl,\Omega) \quad\mbox{ and }\\
 C^{1}_{k} - C_{0}, D^{1}_{k} - D_{0}&: H^{1}(\Omega)\to H^{1}(\Omega)
\end{aligned}
$$   
are compact.
\end{corollary}
Our task is now reduced to finding the essential spectrum of the operators
$\eta A_{0} + \nu B_{0}$ and $\alpha C_{0} + \beta D_{0}$.

\subsection{Analysis of the electric formulation}\label{SS:E12-Anal}
We want to determine the essential spectrum of the operator $\eta A_{0} + \nu B_{0}$ in $\bH(\curl,\Omega)$ for any $\eta,\nu\in\C$. In order to reduce this problem to a better-known one, we use the technique described in \cite{CoDarKo2010,CoDarSak2012} of extension by integration by parts. This transforms the strongly singular VIE into an equivalent system coupling a weakly singular VIE with a boundary integral equation.

We use the definitions \eqref{E:A0}--\eqref{E:B0}, where the derivatives and the convolution commute as operations on functions or distributions on $\R^{2}$. The multiplication by the characteristic function $\chi_{\Omega}$ is understood as the operator of extension of a function defined on $\Omega$ by zero outside of $\Omega$. Derivatives do not commute with $\chi_{\Omega}$, but we have the formulas of integration by parts or weak form of Green's formulas
\begin{align}
\label{E:gamman}
\Div \chi_{\Omega}\bE &= \chi_{\Omega}\Div\bE - \gamma'\bn\cdot\bE 
\quad&&\mbox{ for } \bE\in \bH(\Div,\Omega)\\
\label{E:gammat}
\curl \chi_{\Omega}\bE &= \chi_{\Omega}\curl\bE - \gamma'\bn\times\bE 
\quad&&\mbox{ for } \bE\in \bH(\curl,\Omega)\,.
\end{align}
Here the embedding $\gamma':H^{-\frac12}(\Gamma)\to H^{-1}_{\mathrm{comp}}(\R^{2})$ is the adjoint of the trace operator, and 
$$
\begin{aligned}
\bE\mapsto\bn\cdot\bE&: \;\bH(\Div,\Omega) \to H^{-\frac12}(\Gamma) \quad\mbox{ and }\\
\bE\mapsto\bn\times\bE&: \;\bH(\curl,\Omega) \to H^{-\frac12}(\Gamma)
\end{aligned}
$$
are the normal, respectively tangential, boundary trace mappings. 

\subsubsection{The electric operator}
For $ \bE\in \bH(\Div,\Omega)$ we can now write
\begin{equation}
\label{E:A0onHdiv}
 A_{0}\bE = -\grad \mathcal{N}\Div\bE + \grad \SS \bn\cdot\bE\,.
\end{equation}
Here $\mathcal{N}$ is the Newton potential
$$
 \mathcal{N}u(x) = \int_{\Omega}g_{0}(x-y)\,u(y)\,dy
$$
and $\SS$ is the harmonic single layer potential
$$
  \SS  v(x) = \int_{\Gamma}g_{0}(x-y)\,v(y)\,ds(y)\,.
$$
Using the fact that the Newton potential is a right inverse of $-\Delta$ and that single layer potentials are harmonic on the complement of $\Gamma$, we see that \eqref{E:A0onHdiv} implies for any $\bE\in \bH(\Div,\Omega)$
\begin{equation}
\label{E:divA0}
 \Div A_{0}\bE = \Div\bE.
\end{equation}
We now split $L^{2}(\Omega)^{2}$ into an orthogonal sum
\begin{equation}
\label{E:splitL2}
  L^{2}(\Omega)^{2} = \grad H^{1}_{0}(\Omega) \oplus \bH(\Div0,\Omega) 
\end{equation}
where the second summand is defined by this orthogonality. Later on in Section~\ref{S:sym} we will also need the orthogonal decomposition
\begin{equation}
\label{E:H0curl0}
  L^{2}(\Omega)^{2} = \grad H^{1}(\Omega) \oplus \bH_{0}(\curl0,\Omega) \,.
\end{equation}
We observe that in \eqref{E:splitL2} both summands are invariant subspaces of $A_{0}$:
On $\varphi\in H^{1}_{0}(\Omega)$, $\grad$ commutes with $\chi_{\Omega}$, so that 
$$
\begin{aligned}
 A_{0}\grad\varphi &= -\grad\Div g_{0}*(\chi_{\Omega}\grad\varphi)\\
   &= -\grad\Div\grad g_{0}*(\chi_{\Omega}\varphi)
   = \grad\varphi\,.
\end{aligned}
$$
That is, $A_{0}$ acts as the identity on the first summand in \eqref{E:splitL2}.
According to \eqref{E:divA0}, $A_{0}$ maps the second summand into itself.

Therefore we only need to study the spectral properties of $A_{0}$ on $\bH(\Div0,\Omega)$. Taking the normal trace of \eqref{E:A0onHdiv} and using the jump relations of the single layer potential gives
\begin{equation}
\label{E:nA0}
 \bn \cdot 
 A_{0}\bE = \bn\cdot\grad \SS \bn\cdot\bE\
  = \tfrac12 \bn\cdot\bE + K' \,\bn\cdot\bE
\end{equation}
with the integral operator of the normal derivative of the single layer potential on $\Gamma$
\begin{equation}
\label{E:K'}
  K'u(x) = \int_{\Gamma}\partial_{n(x)}g_{0}(x-y) u(y)\,ds(y)\quad(x\in\Gamma)\,.
\end{equation}
The extension method then consists in considering $\bn\cdot\bE$ as an independent unknown and writing the equation for the resolvent of $A_{0}$
$$
  \lambda\bE - A_{0}\bE = \bF
$$
as the system that we obtained for $\bE\in \bH(\Div0,\Omega)$ and 
$v=\bn\cdot\bE\in H^{-\frac12}(\Gamma)$, where $w=\bn\cdot\bF$:
$$
  \begin{pmatrix}
  \lambda\bE - \grad\SS v\\
  \lambda v - (\frac12\Id +K') v
  \end{pmatrix}
  = \begin{pmatrix} \bF\\ w \end{pmatrix}
$$
Due to its triangular nature, this latter system is easy to analyze: The terms on the diagonal are $\lambda\Id$ and $\lambda\Id - (\frac12\Id +K')$ and their invertibility or Fredholmness determines the corresponding property of $A_{0}$. 

As a tool, we need the essential spectrum of the operator 
$\partial_{n}\SS = \frac12\Id +K'$ in the space $H^{-\frac12}(\Gamma)$. We write  \begin{equation}
\label{E:SK'}
 \Sigma=\Sess(\tfrac12\Id +K') \,.
\end{equation}
The following result is known \cite{CoSte_Banach1985,CoSte_JMAA1985,CoRemPos}
\begin{lemma}
\label{L:SK'}
 Let $\Gamma$ be the boundary of a bounded Lipschitz domain $\Omega\subset\R^{2}$. Then 
 $$  \frac12\Id +K' : H^{-\frac12}(\Gamma)\to H^{-\frac12}(\Gamma) $$
 is a selfadjoint contraction with respect to a certain scalar product in $H^{-\frac12}(\Gamma)$. Its essential spectrum $\Sigma$ is a compact subset of the open interval $(0,1)$. \\
 If $\Gamma$ is smooth (of class $C^{1+\alpha}$, $\alpha>0$), then $K'$ is compact, hence $\Sigma=\{\frac12\}$. \\
 If $\Gamma$ is piecewise smooth with corner 
 angles $\omega_{j}\subset(0,2\pi)$, then $\Sigma$ is an interval containing $\frac12$, 
 namely the convex hull of the numbers 
 $\frac{\omega_{j}}{2\pi}$ and  $1-\frac{\omega_{j}}{2\pi}$.
\end{lemma} 
As a result of the extension method, we obtain the essential spectrum of $A_{0}$.
\begin{theorem}
\label{T:SA0}
 Let $\Omega\subset\R^{2}$ be a bounded Lipschitz domain. The essential spectrum of the operator $A_{0}$ in $L^{2}(\Omega)^{2}$ is given by
 $$
   \Sess(A_{0}) = \{0\}\cup\Sigma\cup\{1\}
 $$
 where $\Sigma$ is defined in \eqref{E:SK'}. 
 The operator $A_{0}$ acts as a bounded operator in each one of the spaces 
 $\bH(\Div,\Omega)$ and $\bH(\curl,\Omega)$. In both spaces, the essential spectrum is the same as in $L^{2}(\Omega)^{2}$.
\end{theorem}
\begin{proof}
A quick way to prove that the extension method described above does indeed lead to a characterization of the essential spectrum of $A_{0}$ is to use the well-known ``recombination lemma'', namely that for two linear operators
$$
  S: Y\to X \mbox{ and } T: X\to Y
$$ 
the two products $ST$ and $TS$ are spectrally equivalent outside zero. This means that for $\lambda\ne0$, the two operators $\lambda\Id_{X}-ST$ and $\lambda\Id_{Y}-TS$ are simultaneously Fredholm or not and have the same kernel and cokernel dimensions. For a proof, see \cite[Lemma 4.1]{CoDarSak2012} or \cite[p. 38]{Gohberg1990}.

In our case, we factorize $A_{0}$ in $\bH(\Div0,\Omega)$ as $A_{0}=ST$ with
\begin{align*}
 S&: H^{-\frac12}(\Gamma) \to \bH(\Div0,\Omega)\,;\quad
 S v =  \grad\SS v\,,\\
 T&: \bH(\Div0,\Omega) \to H^{-\frac12}(\Gamma) \,;\quad
 T\bE = \bn\cdot\bE \,.
\end{align*}
Thus $A_{0}$ is spectrally equivalent outside zero to 
$$
 TS= \bn\cdot\grad \SS =\frac12\Id +K'\;\mbox{ on } H^{-\frac12}(\Gamma)\,.
$$
That $\lambda=0$ is in the essential spectrum follows from \eqref{E:A0onHdiv} directly:
$A_{0}=0$ on the space
$$
  \bH_{0}(\Div0,\Omega) = \{\bE\in\bH(\Div,\Omega) \mid \Div\bE=0, \bn\cdot\bE=0\}\,.
$$ 
Thus on $\bH(\Div0,\Omega)$, the essential spectrum is $\{0\}\cup\Sigma$, and together with the decomposition \eqref{E:splitL2} and $A_{0}\vert_{\grad H^{1}_{0}(\Omega)}=\Id$, we obtain $\Sess(A_{0}) = \{0\}\cup\Sigma\cup\{1\}$.
The essential spectrum is the same in $\bH(\Div,\Omega)$ and in $\bH(\curl,\Omega)$, because both spaces are invariant under $A_{0}$, and the operator $S$ defined above maps $H^{-\frac12}(\Gamma)$ in fact into $\bH(\curl,\Omega)\cap\bH(\Div0,\Omega)$.

\end{proof}

\subsubsection{The magnetic operator}\label{SSS:mag}
We can proceed along the same lines for the operator $B_{0}$ defined in \eqref{E:B0}. Using the integration by parts formula \eqref{E:gammat}, we get for $ \bE\in \bH(\curl,\Omega)$
\begin{equation*}
\begin{split}
  B_{0}\bE &= \vcurl \curl g_{0}* (\chi_{\Omega}\bE) + \vcurl \SS \bn\times\bE\\
    &=\bE - A_{0}\bE + \vcurl \SS \bn\times\bE
\end{split}
\end{equation*}
If $\bE\in \bH(\curl,\Omega)\cap \bH(\Div,\Omega)$, we can further reduce this using \eqref{E:A0onHdiv} to
\begin{equation}
\label{E:B0onHcurl}
  B_{0}\bE =\bE +\grad \mathcal{N}\Div\bE - \grad \SS \bn\cdot\bE + \vcurl \SS \bn\times\bE\,.
\end{equation}
If we want to use the extension technique like we did for $A_{0}$, we can reduce $B_{0}$ to an equivalent boundary integral equation, but we now need both the normal and tangential traces on the boundary. For this we use the orthogonal sum
\begin{equation}
\label{E:splitLHcurl}
  \bH(\curl,\Omega) = \grad H^{1}_{0}(\Omega) \oplus \bX \quad\mbox{ with } 
  \bX=\bH(\curl,\Omega) \cap \bH(\Div0,\Omega) \,.
\end{equation}
Note that this decomposition is orthogonal both with respect to the $L^{2}$ scalar product and the $\bH(\curl)$ scalar product. 

On $\grad H^{1}_{0}(\Omega)$ we have $B_{0}=0$ by definition. On $\bX$, we use \eqref{E:B0onHcurl} and the recombination lemma and factorize $\Id-B_{0}=ST$, where $S$ and $T$ are now defined as
\begin{align*}
 S&: H^{-\frac12}(\Gamma)^{2} \to \bX\,;\quad
 S \binom{v}{w} =  \grad\SS v-\vcurl\SS w\,,\\
 T&: \bX \to H^{-\frac12}(\Gamma)^{2} \,;\quad
 T\bE = \binom{\bn\cdot\bE}{\bn\times\bE} \,.
\end{align*}
The recombined operator $TS$ is a $2\times2$ matrix of boundary integral operators involving the normal and tangential derivatives of the single layer potential:
$$
 \tfrac12\Id + K' = \bn\cdot\grad\SS  = -\bn\times\vcurl\SS 
$$
and
$$
  \partial_{t}\SS  = \bn\times\grad\SS  = \bn\cdot\vcurl\SS 
$$
Thus $\Id-B_{0}$ on $\bX$ will be spectrally equivalent outside zero to the matrix operator on  $H^{-\frac12}(\Gamma)^{2}$
\begin{equation}
\label{E:TSforB}
 TS=\begin{pmatrix}
     \frac12\Id+K' & -\partial_{t}\SS \\
     \partial_{t}\SS & \frac12\Id+K'
    \end{pmatrix}
\end{equation}
Even for a smooth boundary $\Gamma$, the essential spectrum of this matrix of operators is not obvious, because whereas $K'$ is then compact, the strongly singular integral operator $\partial_{t}\SS$ is not compact. But for smooth $\Gamma$ we can use symbols of pseudodifferential operators to find the essential spectrum of the operator matrix, and for piecewise smooth $\Gamma$ we can use Mellin transformation to the same end, but we have not seen an obvious way to find the essential spectrum of $TS$ for a general bounded Lipschitz boundary $\Gamma$. 

We will instead turn around the equivalence between $\Id-B_{0}$ and $TS$ by finding a different way to analyze $B_{0}$ and thence deduce, as a corollary, the spectrum of $TS$ in \eqref{E:TSforB}.

We go back to the definition \eqref{E:B0} of $B_{0}$ and use $\curl\vcurl=-\Delta$ to obtain $\curl B_{0} \bE = \chi_{\Omega}\curl\bE$.
On the space $\bH(\curl,\Omega)$, we have therefore 
\begin{equation}
\label{E:curlB0}
 \curl B_{0} \bE = \curl\bE\,.
\end{equation}
This implies
$$
  B_{0}\bE = \vcurl g_{0}* (\chi_{\Omega}\curl\bE)
            = \vcurl g_{0}* (\chi_{\Omega}\curl B_{0}\bE)
            = B_{0}^{2}\bE\,.
$$
\begin{proposition}
\label{P:B0projector}
 The operator $B_{0}$ on $\bH(\curl,\Omega)$ is a projection operator.
\end{proposition}
Both the kernel of $B_{0}$, which is $\bH(\curl0,\Omega)$, and the image of $B_{0}$, which is the orthogonal complement of $\bH(\curl0,\Omega)$, hence contains 
$\vcurl H^{1}_{0}(\Omega)\cap\bH(\curl,\Omega)$, are of infinite dimension. 
Therefore the essential spectrum is equal to the spectrum:
$$
  \Sess(B_{0}) = \{0,1\}\,.
$$
As a corollary, we get the same essential spectrum for the matrix operator $TS$ in \eqref{E:TSforB}, now valid for any bounded Lipschitz $\Gamma$. But by a careful inspection of the recombination argument we can do more: We have $B_{0}S=0$ and therefore
$$
  (TS)^{2} = T(ST)S = T(\Id-B_{0})S = TS\,.
$$
\begin{corollary}
\label{C:TSprojector}
 For any bounded Lipschitz domain $\Omega\subset\R^{2}$ with boundary $\Gamma$, the matrix $TS$ of boundary integral operators in \eqref{E:TSforB} is a projection in $H^{-\frac12}(\Gamma)^{2}$.
\end{corollary}
This result is equivalent to the following two interesting relations for the tangential derivative of the single layer potential (Hilbert transform) on a Lipschitz boundary:
\begin{equation}
\label{E:TSrelations}
 (\partial_{t}\SS)^{2} = \tfrac14\Id-(K')^{2} \quad\mbox{ and }\quad
 \partial_{t}\SS\, K' + K'\,\partial_{t}\SS = 0 \,.
\end{equation}

\subsubsection{The electromagnetic operator}
We combine now the analysis for the operators $A_{0}$ and $B_{0}$ and find the essential spectrum of $\eta A_{0}+ \nu B_{0}$. As before, we use the orthogonal splitting \eqref{E:splitLHcurl} into electrostatic fields and electrodynamic waves. 
On the first summand $\grad H^{1}_{0}(\Omega)$, we have
$\eta A_{0}+ \nu B_{0}=\eta\Id$. Hence $\eta\in\Sess(\eta A_{0}+ \nu B_{0})$, and we only need to consider the operator in the second summand $\bX$. We use the extension method, and after seeing how this worked for the operator $B_{0}$, we do not use the extension by the normal and tangential traces on the boundary, but by $\curl\bE$ in $\Omega$ and $\bn\cdot\bE$ on $\Gamma$. Thus we write for $\bE\in\bX$
\begin{equation}
\label{E:eA+nB}
 \eta A_{0} + \nu B_{0} = \eta \grad \SS \bn\cdot\bE 
                        + \nu \vcurl g_{0}* (\chi_{\Omega}\curl\bE)
\end{equation}
We see that we can factorize $\eta A_{0} + \nu B_{0} =ST$, where $S$ and $T$ are now defined as
\begin{align*}
 S&: H^{-\frac12}(\Gamma) \times L^{2}(\Omega) \to \bX\,;\quad
 S \binom{v}{F} = \eta \grad\SS v + \nu \vcurl\mathcal{N}F \,,\\
 T&: \bX \to H^{-\frac12}(\Gamma) \times L^{2}(\Omega)\,;\quad
 T\bE = \binom{\bn\cdot\bE}{\curl\bE} \,.
\end{align*}
The recombined operator $TS$ now acts on $H^{-\frac12}(\Gamma) \times L^{2}(\Omega)$ and is given by
$$
 TS\binom{v}{F} = 
   \binom{\eta \bn\cdot\grad\SS v + \nu \bn\cdot\vcurl \mathcal{N}F}{\nu\curl\vcurl\mathcal{N}F} =
   \binom{\eta(\tfrac12\Id+K')v + \nu\bn\cdot\vcurl\mathcal{N}F}{\nu F}
$$
The essential spectrum of the triangular matrix
$$
  TS = \begin{pmatrix}
         \eta(\tfrac12\Id+K') & \nu \bn\cdot\vcurl \mathcal{N}\\
         0 & \nu\Id
       \end{pmatrix}
$$
is obvious: $\Sess(TS)=\eta\Sigma\cup\{\nu\}$. To this we might have to add the point $\{0\}$. 

Let us show that $0$ is not in the essential spectrum if $\eta\ne0$ and $\nu\ne0$. 
We know that $0$ is then not in the essential spectrum of $TS$ because $0\not\in\Sigma$. Thus $TS$ is a Fredholm operator. One can see that $T$ is also a Fredholm operator, because the boundary value problem
$$
\begin{aligned}
 \Div\bE&=0 \;\mbox{ and }\curl\bE = F\; \mbox{ in }\Omega\\
 \bn\cdot\bE&=v \;\mbox{ on }\Gamma
\end{aligned}
$$
has, for any $F\in L^{2}(\Omega)$, $v\in H^{-\frac12}(\Gamma)$, a solution $\bE\in\bX$, as soon as $\int_{\Gamma_{j}}v\,ds=0$ for all connected components $\Gamma_{j}$ of $\Gamma$. The solution is unique up to a finite-dimensional kernel. For a detailed proof of the corresponding result in $\R^{3}$, see \cite[Prop. 3.14]{ABDG}. 
Now if $TS$ and $T$ are Fredholm, then $S$ is Fredholm, too. Hence $ST$ is a Fredholm operator, and $0\not\in\Sess(ST)$.

We have proved the main result of this section
\begin{theorem}
\label{T:SessA0+B0}
Let $\Omega\subset\R^{2}$ be a bounded Lipschitz domain. For any $\eta,\nu\in\C$, the essential spectrum of the operator $\eta A_{0} + \nu B_{0}$ in the space $\bH(\curl,\Omega)$ is given by
$$
  \Sess(\eta A_{0} + \nu B_{0}) = \{\eta\} \cup \eta\Sigma \cup\{\nu\}\,.
$$
\end{theorem}
\begin{corollary}
\label{C:A+BFredholm}
 Let $\Omega\subset\R^{2}$ be a bounded Lipschitz domain. Let the coefficients $\epsilon_{r}$ and $\mu_{r}$ be complex constants. Then the integral operator of the VIE \eqref{E:VIE-TE-E12} is a Fredholm operator in $\bH(\curl,\Omega)$ if and only if 
 $$
  \epsilon_{r}\ne0\,,\quad \mu_{r}\ne0\quad\mbox{ and }\quad \epsilon_{r}\ne1-\frac1\sigma \mbox{ for all }\sigma\in\Sigma\,.
$$
If the boundary $\Gamma$ is smooth, then the condition is
$$
  \epsilon_{r}\not\in\{0,-1\}\quad\mbox{ and }\quad  \mu_{r}\ne0\,.
$$
\end{corollary}
\begin{proof}
 The operator $\Id-A^{\eta}_{k}-B^{\nu}_{k}$ is Fredholm if and only if 
 $1\not\in\Sess(A^{\eta}_{k}-B^{\nu}_{k})=\Sess(\eta A_{0}+\nu B_{0})$.
 According to the theorem, this is equivalent to the conditions
$$
  1\ne\eta, \; 1\ne\nu,\; 1\ne\eta\sigma\;\mbox{ for all }\sigma\in\Sigma\,.
$$
With the definitions \eqref{E:eta-nu}, this is easily seen to be equivalent to the stated conditions on $\epsilon_{r}$ and $\mu_{r}$. For smooth $\Gamma$, we use the fact that $\Sigma=\{\frac12\}$.
\end{proof}

\subsection{Analysis of the magnetic formulation}\label{SS:H3-Anal}
In this paragraph, we study Fredholm properties of the VIE \eqref{E:VIE-TE-H3} for the magnetic field $H_{3}$ in $H^{1}(\Omega)$. Since the magnetic formulation describes the same TE transmission problem \eqref{E:TE} as the electric formulation studied in the previous paragraph, we expect equivalent results to those given in Theorem~\ref{T:SessA0+B0} and Corollary~\ref{C:A+BFredholm}. 

Whereas previously we had a two-dimensional Maxwell system that in some respects closely resembles the full three-dimensional Maxwell system, we have now, however, a scalar Helmholtz problem that is specific to the two-dimensional situation. For the analysis of the scalar VIE we use similar techniques as before, namely the partial integration and extension method, but we can take advantage of the simpler structure and allow more general coefficients. In the following, we assume 
\begin{equation}
\label{E:ab}
 \alpha\in C^{1}(\overline\Omega)\,;\quad \beta\in L^{\infty}(\Omega)\,.
\end{equation} 
Note that according to the definition \eqref{E:defab} of $\alpha$, this already includes the condition $\epsilon_{r}\ne0$ on $\overline\Omega$.
Recall that in the absence of further specification, we assume that $\Omega\subset\R^{2}$ is a bounded Lipschitz domain. 

According to Lemma \ref{L:CDinH1}, the operator $D^{\beta}_{k}$ is compact. This implies that it, and with it the coefficient $\beta$, hence the magnetic permeability $\mu$, will disappear from the further discussion. It has no influence on the Fredholm properties of the TE problem. Note that also in the previous paragraph the only condition on $\mu$ we encountered  was $\mu_{r}\ne0$. This was simply a necessary prerequisite for the definition of the coefficient $\nu$ in the electric VIE \eqref{E:VIE-TE-E12}. The condition $\nu\ne1$ that is required according to Theorem~\ref{T:SessA0+B0} is always satisfied if $\nu$ is defined from $\mu_{r}$ as in \eqref{E:eta-nu}.

\begin{theorem}
\label{T:TE-H3-Fred}
Let $\Omega\subset\R^{2}$ be a bounded Lipschitz domain. Let the coefficients $\alpha$ and $\beta$ be defined by
$$
 \alpha=1-\tfrac1{\epsilon_{r}}, \qquad  \beta=\mu_{r}-1
$$
and satisfy the conditions \eqref{E:ab}. Then the operator 
$\Id-C_{k}^{\alpha}-D_{k}^{\beta}$ of the VIE \eqref{E:VIE-TE-H3} is Fredholm in $H^{1}(\Omega)$ if and only if the boundary integral operator
$$
  (1-\tfrac\alpha2)\Id - \alpha K: H^{1/2}(\Gamma)\to H^{1/2}(\Gamma)
$$
is Fredholm, where $K$ is the boundary integral operator of the double layer potential, see \eqref{E:K} below.\\ 
A sufficient condition for this is that for all $x\in\Gamma$ and $\sigma\in\Sigma$ (defined in \eqref{E:SK'})
$$
  \epsilon_{r}(x) \ne \frac{\sigma}{\sigma-1}\,.
$$
If $\epsilon_{r}$ is constant, then this condition is also necessary.\\
If $\Gamma$ is smooth, the necessary and sufficient condition is $\epsilon_{r}(x)\ne-1$ for all $x\in\Gamma$.  
\end{theorem}
\begin{proof}
Our task is reduced to the determination of the Fredholm properties of the operator $\Id-C_{k}^{\alpha}$ defined in \eqref{E:C}:
$$
  C_{k}^{\alpha}u = \curl g_{k}*(\alpha\chi_{\Omega}\vcurl u)\,.
$$
We use the partial integration formula (dual to \eqref{E:gammat}) 
\begin{equation}
\label{E:}
 \vcurl \chi_{\Omega}u = \chi_{\Omega}\vcurl u + \gamma'\bt \gamma u 
\quad\mbox{ for } u\in H^{1}(\Omega)\,,
\end{equation}
where $\bt$ is the unit tangent vector and $\gamma u$ is the trace of $u$ on $\Gamma$.
With the product rule $\alpha\vcurl u=\vcurl(\alpha u) - u\vcurl\alpha$ and the identity 
$\curl g_{k}* \vcurl = -\Delta g_{k}* = \Id +k^{2}g_{k}*$, we can write
$$
\begin{aligned}
  C_{k}^{\alpha}u &= 
   \curl g_{k}*(\alpha\vcurl\chi_{\Omega} u) -\curl g_{k}*(\gamma'\bt\gamma\alpha u)\\
   &= \curl g_{k}*(\vcurl(\alpha\chi_{\Omega} u)) 
      -\curl g_{k}*(\chi_{\Omega} u\vcurl\alpha))
      - \curl g_{k}*(\gamma'\bt \gamma \alpha u)\\
   &= \alpha u + C_{1} u + \DD_{k}(\gamma\alpha u)\,.
\end{aligned}  
$$
Here $C_{1}$ is the weakly singular integral operator given by
$$
  C_{1} u(x) = k^{2}\int_{\Omega} g_{k}(x-y)\alpha(y)u(y)\,dy
    -\curl \int_{\Omega} g_{k}(x-y) u(y)\vcurl\alpha(y)\,dy\,.
$$
It maps $L^{2}(\Omega)$ boundedly to $H^{1}(\Omega)$ and is therefore a compact operator in $H^{1}(\Omega)$. The operator $\DD_{k}$ is the Helmholtz double layer potential
$$
  -\curl g_{k}*(\gamma'\bt \gamma \alpha u)(x) = \DD_{k}(\gamma\alpha u)(x) = \int_{\Gamma} \partial_{n(y)}g_{k}(x-y)\alpha(y)u(y)\,dy\quad (x\in\Omega)\,.
$$
Thus 
$$
 (\Id - C_{k}^{\alpha}) u = (1-\alpha) u - C_{1} u - \DD_{k}(\gamma\alpha u)\,.
$$
Modulo compact operators, we are left with the operator
$$
 u\mapsto (1-\alpha)u - \DD_{0}(\gamma\alpha u) 
 = \tfrac1{\epsilon_{r}}\big(u-\epsilon_{r}\DD_{0}(\gamma\alpha u)\big)\,,
$$
where $\DD_{0}$ is the harmonic double layer potential.
Since by assumption $\epsilon_{r}\ne0$ on $\overline\Omega$, we only need to consider the operator $\Id - \epsilon_{r}\DD_{0}(\gamma\alpha \cdot)$. By the recombination lemma, the Fredholm properties of this operator on $H^{1}(\Omega)$ are the same as those of the operator
$$
  \Id - \gamma(\alpha\epsilon_{r}\DD_{0}) \; \mbox{ on } H^{\frac12}(\Gamma)\,.
$$
Here we used the factorization $\epsilon_{r}\DD_{0}(\gamma\alpha \cdot)=ST$ with
\begin{align*}
 S&: H^{\frac12}(\Gamma) \to H^{1}(\Omega)\,;\quad
 S{v} =  \epsilon_{r}\DD_{0} v\,,\\
 T&: H^{1}(\Omega) \to  H^{\frac12}(\Gamma) \,;\quad
 Tu  = \gamma(\alpha u) \,.
\end{align*}
We then use the jump relation $\gamma\DD_{0}=-\frac12\Id+K$ for the double layer potential with
\begin{equation}
\label{E:K}
 Kv(x) = \int_{\Gamma} \partial_{n(y)}g_{0}(x-y)v(y)\,ds(y) \quad(x\in\Gamma)\,.
\end{equation}
Thus
$$
 \Id - \gamma(\alpha\epsilon_{r}\DD_{0}) =
 \Id + \tfrac{\alpha\epsilon_{r}}2\Id - \alpha\epsilon_{r}K
 = \epsilon_{r}\big( (\tfrac1{\epsilon_{r}}+\tfrac\alpha2)\Id - \alpha K\big)
$$
Now $\frac1{\epsilon_{r}}+\tfrac\alpha2=1-\tfrac\alpha2$, and we see that we have to determine the Fredholm properties of the boundary integral operator
$$
 G:= (1-\tfrac\alpha2)\Id - \alpha K = \Id - \alpha(\tfrac12\Id + K)
$$
as claimed.
Since $K$ is the adjoint operator of $K'$, we have
$$
  \Sess(\tfrac12\Id+K) = \Sess(\tfrac12\Id+K')=\Sigma\,.
$$
If $\alpha$ is constant on $\Gamma$, we see that $G$ is Fredholm if and only if $\alpha\sigma\ne1$ for all $\sigma\in\Sigma$. This is equivalent to the stated condition 
$\epsilon_{r} \ne \frac{\sigma}{\sigma-1}$. 
If $\alpha$ is not constant on $\Gamma$, we can use a standard localization procedure based on the fact that the commutator between the singular integral operator $K$ and the operator of multiplication with a smooth function is compact in $H^{\frac12}(\Gamma)$. If the condition $\epsilon_{r}(x) \ne \frac{\sigma}{\sigma-1}$ for all $\sigma\in\Sigma$ is satisfied, one can construct a local regularizer using an inverse modulo compact operators for the operator with $\epsilon_{r}$ frozen in the point $x$. These local regularizers are then patched together using a partition of unity, giving a global regularizer that shows that $G$ is Fredholm. We omit the technical details here, see \cite[Section 3]{sakly2014}. 
Finally, if $\Gamma$ is smooth, then $K'$ is compact, therefore $G$ is Fredholm if and only if $(1-\tfrac\alpha2)\Id $ is. This is the case if and only if $\alpha(x)\ne2$ for all $x\in\Gamma$, which is equivalent to $\epsilon_{r}(x)\ne-1$.
\end{proof}

\section{Symmetry of the spectrum}\label{S:sym}

The condition on $\epsilon_{r}$ in Theorem~\ref{T:TE-H3-Fred}: $\epsilon_{r}\ne\frac{\sigma}{\sigma-1}$ is not the same as the one given in Corollary~\ref{C:A+BFredholm}:
$\epsilon_{r}\ne1-\frac1\sigma$.
 They coincide only if 
$$
  \sigma\in\Sigma \quad \Longleftrightarrow\quad 1-\sigma\in\Sigma.
$$
This symmetry of the essential spectrum of the operator $\frac12\Id+K'$ with respect to the point $\frac12$ is a non-trivial result that may be new. It can in fact be seen as a consequence of the fact that both VIEs considered in Corollary~\ref{C:A+BFredholm} via the electric field formulation and in Theorem~\ref{T:TE-H3-Fred} via the magnetic field formulation are equivalent to the same Maxwell transmission problem. Since this amounts to a very indirect way of proving the result, we shall give an independent proof below that does not use volume integral equations. Note that the result is known for smooth domains where $\Sigma=\{\frac12\}$ and for piecewise smooth domains where Mellin analysis shows that $\Sigma$ is an interval with midpoint $\frac12$, but we show it for arbitrary bounded Lipschitz domains in $\R^{2}$. It is probably not true in higher dimensions.

\begin{theorem}
\label{T:sym}
 Let $\Omega\subset\R^{2}$ be a bounded domain with a Lipschitz boundary $\Gamma$. Let $K$ be the boundary integral operator of the harmonic double layer potential acting on $H^{\frac12}(\Gamma)$. Then 
\begin{equation}
\label{E:sym}
  \Sess(K) = \Sess(-K)\,.
\end{equation}
\end{theorem}
This symmetry implies that the operators $\frac12\Id+K$ and $\frac12\Id-K$ in $H^{\frac12}(\Gamma)$ and the operators $\frac12\Id+K'$ and $\frac12\Id-K'$ in $H^{-\frac12}(\Gamma)$ all have the same essential spectrum $\Sigma$, and the same holds for the Helmholtz double layer potential.

The proof uses the equivalence between the boundary integral equation and a scalar transmission problem together with the following known symmetry result for scalar elliptic boundary value problems.
\begin{lemma}
\label{L:symeps}
 Let $\Omega\subset\R^{2}$ be a bounded Lipschitz domain and let the complex-valued coefficient function $\epsilon\in L^{\infty}(\Omega)$ satisfy $\frac1\epsilon\in L^{\infty}(\Omega)$. Then
\begin{equation}
\label{E:divepsgrad}
   \Div\epsilon\grad:\; H^{1}_{0}(\Omega) \to H^{-1}(\Omega)
\end{equation}
is Fredholm if and only if
\begin{equation}
\label{E:div1/epsgrad}
   \Div\tfrac1\epsilon\grad:\; H^{1}(\Omega) \to \big(H^{1}(\Omega)\big)'
\end{equation}
is Fredholm.
\end{lemma}
This lemma is due to \cite[Theorem 4.6]{BBChesnelCiarlet_2D2014}, where it was proved under the hypotheses that $\Omega$ is simply connected and that $\epsilon$ is real-valued. We give an independent proof that does not need these additional assumptions. Note that, as $\Delta:H^{1}_{0}(\Omega)\to H^{-1}(\Omega)$ is an isomorphism, 
the operator in \eqref{E:divepsgrad} is Fredholm if and only if the mapping $u\mapsto w$ defined by
\begin{equation}
\label{E:epsvar}
 \int_{\Omega}\epsilon\grad u\cdot\grad v\,dx = \int_{\Omega}\grad w\cdot\grad v\,dx
 \quad\mbox{ for all }v\in H^{1}_{0}(\Omega)
\end{equation}
is Fredholm in $H^{1}_{0}(\Omega)$. Similarly, the operator in \eqref{E:div1/epsgrad} is Fredholm if and only if the mapping $u\mapsto w$ defined by
\begin{equation}
\label{E:1/epsvar}
 \int_{\Omega}\tfrac1\epsilon\grad u\cdot\grad v\,dx 
  = \int_{\Omega}\grad w\cdot\grad v\,dx
 \quad\mbox{ for all }v\in H^{1}(\Omega)
\end{equation}
is Fredholm in $H^{1}(\Omega)$.
\begin{proof}
We use a simple linear algebra lemma whose elementary proof we leave to the reader.
\begin{lemma}
\label{L:ST}
Let $X$, $Y$, $Z$ be vector spaces and $S:Y\to Z$, $T:X\to Y$ be linear operators such that $S$ is surjective and $T$ is injective. Then 
$T$ induces an isomorphism between $\ker(ST)$ and $\ker S\cap \im T$, and 
$S$ induces an isomorphism between the quotient spaces 
$Y/(\ker S+\im T)$ and $Z/\im(ST)$. In particular, $ST:X\to Z$ is an isomorphism if and only if $Y=\ker S\oplus\im T$, and $ST$ is a Fredholm operator if and only if
$$
  \dim(\ker S\cap \im T)<\infty \quad\mbox{ and }
  \codim(\ker S+\im T)<\infty\,.
$$ 
\end{lemma}
We apply this lemma to the situation
$$
 \Div: L^{2}(\Omega)^{2} \to H^{-1}(\Omega)\,,\quad
 \epsilon\grad: H^{1}_{0}(\Omega) \to L^{2}(\Omega)^{2}
$$
and obtain that 
$\Div\epsilon\grad:\; H^{1}_{0}(\Omega) \to H^{-1}(\Omega)$ is Fredholm if and only if 
$$
  \dim(\bH(\Div0,\Omega)\cap\epsilon\!\grad H^{1}_{0}(\Omega)) <\infty
$$
$$
  \;\mbox{ and }
  \codim(\bH(\Div0,\Omega)+\epsilon\!\grad H^{1}_{0}(\Omega)) <\infty
  \;\mbox{ in }L^{2}(\Omega)^{2}\,.
$$
Using the well known fact that there is a finite dimensional space $\mathcal{H}_{0}$ such that 
$\bH(\Div0,\Omega)=\vcurl H^{1}(\Omega)+\mathcal{H}_{0}$, we see that this condition is equivalent to the fact that
\begin{equation}
\label{E:curlH1+epsgradH1}
 \dim(\vcurl H^{1}(\Omega)\cap\epsilon\!\grad H^{1}_{0}(\Omega)) <\infty
  \;\mbox{ and }
  \codim(\vcurl H^{1}(\Omega)+\epsilon\!\grad H^{1}_{0}(\Omega)) <\infty\,.
\end{equation}
Dividing by $\epsilon$, we see that the latter condition is equivalent to
$$
 \dim(\tfrac1\epsilon\vcurl H^{1}(\Omega)\cap\grad H^{1}_{0}(\Omega)) <\infty
  \;\mbox{ and }
  \codim(\tfrac1\epsilon\vcurl H^{1}(\Omega)+\grad H^{1}_{0}(\Omega)) <\infty\,.
$$
Using the definition \eqref{E:H0curl0}, it is also well known that there exists a finite dimensional space $\mathcal{H}_{1}$ such that
$\bH_{0}(\curl0,\Omega)=\grad H^{1}_{0}(\Omega)+\mathcal{H}_{1}$,  and therefore we obtain the equivalent condition
$$
 \dim(\tfrac1\epsilon\vcurl H^{1}(\Omega)\cap \bH_{0}(\curl0,\Omega)) <\infty
  \;\mbox{ and }
  \codim(\tfrac1\epsilon\vcurl H^{1}(\Omega)+ \bH_{0}(\curl0,\Omega)) <\infty\,.
$$
According to Lemma \ref{L:ST}, applied to the situation
$$
 \curl: L^{2}(\Omega)^{2} \to \big(H^{1}(\Omega)\big)'\,,\quad
 \tfrac1\epsilon\vcurl: H^{1}(\Omega) \to L^{2}(\Omega)^{2}
$$
the latter condition is equivalent to the Fredholmness of the operator
$$
  \curl\tfrac1\epsilon\vcurl: H^{1}(\Omega)\to \big(H^{1}(\Omega)\big)'\,.
$$
We finally note that in two dimensions,
$$
  \curl\tfrac1\epsilon\vcurl = -\Div\tfrac1\epsilon\grad\,.
$$
Therefore the last condition expresses the fact that the operator in \eqref{E:div1/epsgrad} is Fredholm.
\end{proof}
In order to apply this result to our boundary integral operator, we consider the following situation: $\Omega=\Omega^{-}\cup\Gamma\cup\Omega^{+}$, where the interior interface $\Gamma$ is Lipschitz and satisfies
$\Gamma=\partial\Omega^{-}=\partial\Omega^{+}\setminus\partial\Omega$. The coefficient function $\epsilon$ is piecewise constant, $\epsilon=1$ in $\Omega^{+}$, $\epsilon=\epsilon_{r}\in\C$ in $\Omega^{-}$. Define $K'$ as the boundary integral operator of the normal derivative of the single layer potential, adjoint of the double layer potential, acting in the space $H^{-\frac12}(\Gamma)$, see \eqref{E:K'}. In this situation, we have the following result. 
\begin{lemma}
\label{L:Var-BIE}
 The operator $\Div\epsilon\grad$, acting either from $H^{1}_{0}(\Omega)$ to $H^{-1}(\Omega)$ or from $H^{1}(\Omega)$ to $\big(H^{1}(\Omega)\big)'$, is Fredholm if and only if the integral operator 
$$
  \frac{\epsilon_{r}+1}2\Id + (\epsilon_{r}-1)K'
$$
is a Fredholm operator in $H^{-\frac12}(\Gamma)$.
\end{lemma}
\begin{proof}
We consider the first case (Dirichlet condition) of $\Div\epsilon\grad$ defined on $H^{1}_{0}(\Omega)$. The second case (Neumann condition) is treated very similarly, the space $H^{1}_{0}(\Omega)$ being replaced by $H^{1}_{*}(\Omega)$, the space of $H^{1}(\Omega)$ functions of mean value zero.
We use the variational formulation \eqref{E:epsvar} for the operator $\Div\epsilon\grad$ on $H^{1}_{0}(\Omega)$. We split the space $H^{1}_{0}(\Omega)$ according to
$$
  H^{1}_{0}(\Omega) = H^{1}_{0,\Gamma}(\Omega) \oplus \mathcal{H}^{1}_{0,\Gamma}(\Omega)\,,
$$
where
$$
 H^{1}_{0,\Gamma}(\Omega) = \{u\in H^{1}_{0}(\Omega) \mid u_{|\Gamma}=0\}
   = H^{1}_{0}(\Omega^{-})\oplus H^{1}_{0}(\Omega^{+}) 
$$
and the second term $\mathcal{H}^{1}_{0,\Gamma}(\Omega)$ is defined as the orthogonal complement of $H^{1}_{0,\Gamma}(\Omega)$ in $H^{1}_{0}(\Omega)$. Orthogonality is understood in the sense of the $H^{1}_{0}$ scalar product appearing in \eqref{E:epsvar}. 

The subspace $H^{1}_{0,\Gamma}(\Omega)$ is invariant under multiplication by $\epsilon$, and it is also invariant under the operator $\Delta^{-1}\Div\epsilon\grad$ defined by the variational formulation \eqref{E:epsvar}. This operator is invertible on this subspace as soon as $\epsilon_{r}\ne0$. Indeed, if we decompose orthogonally 
$$
  u = u_{-}+u_{+}+u_{0}\;, \quad w = w_{-}+w_{+}+w_{0}
$$
with $u_{\pm},w_{\pm}\in H^{1}_{0}(\Omega^{\pm})$ and 
$u_{0},w_{0}\in \mathcal{H}^{1}_{0,\Gamma}(\Omega)$, then we can partially solve \eqref{E:epsvar}:
$$
  u_{+}=w_{+}\,,\quad u_{-}=\tfrac1\epsilon w_{-}\,.
$$
This shows that our operator $\Div\epsilon\grad$ is Fredholm if and only if in the restriction of the variational formulation \eqref{E:epsvar} to $\mathcal{H}^{1}_{0,\Gamma}(\Omega)$
\begin{equation}
\label{E:epsvarharm}
 \int_{\Omega}\epsilon\grad u_{0}\cdot\grad v\,dx = \int_{\Omega}\grad w_{0}\cdot\grad v\,dx
 \quad\mbox{ for all }v\in \mathcal{H}^{1}_{0,\Gamma}(\Omega)
\end{equation}
the mapping $u_{0}\mapsto w_{0}$ is a Fredholm operator in $\mathcal{H}^{1}_{0,\Gamma}(\Omega)$.
From the definition of $\mathcal{H}^{1}_{0,\Gamma}(\Omega)$ as orthogonal complement we find the characterization of this space
$$
  \mathcal{H}^{1}_{0,\Gamma}(\Omega) = \{u_{0}\in H^{1}_{0,\Gamma}(\Omega) \mid
    \Delta u_{0} = 0 \mbox{ in } \Omega^{-}\cup\Omega^{+}\}\,.
$$
The standard representation formula shows that the functions in this space are represented as single layer potentials
$$
  u_{0}(x)=\SS_{D}\phi(x) = \int_{\Gamma} G_{D}(x,y)\phi(y)\,ds(y)\,.
$$
Here $G_{D}$ is the Green function of the domain $\Omega$, that is the fundamental solution of $-\Delta$ satisfying homogeneous Dirichlet boundary conditions on the outer boundary $\partial\Omega$. It is a classical result that $G_{D}(x,y)=g_{0}(x-y)+h_{D}(x,y)$ with a $C^{\infty}(\Omega\times\Omega)$ function $h_{D}$. The jump relations for the single layer potential $\SS_{D}$ are then the same as for the single layer potential $\SS$, in particular for the one-sided traces of the normal derivatives $\partial_{n}^{+}$ and $\partial_{n}^{-}$ on $\Gamma$ 
$$
  \partial_{n}^{\pm} \SS_{D}\phi = \mp\tfrac12\phi + K'_{D}\phi
$$
where $K'_{D}-K'$ is a compact operator on $H^{-\frac12}(\Gamma)$, due to the smoothness of $h_{D}$. The representation of $u_{0}$ by a single layer potential defines an isomorphism between $H^{-\frac12}(\Gamma)$ and $\mathcal{H}^{1}_{0,\Gamma}(\Omega)$, with inverse given by the jump $\phi=\partial_{n}^{-}u_{0}-\partial_{n}^{+}u_{0}$.

We can now use Green's formula in $\Omega^{+}$ and $\Omega^{-}$ for \eqref{E:epsvarharm} and see that this variational formulation is equivalent to the transmission condition
$$
  \epsilon_{r}\partial_{n}^{-}u_{0}-\partial_{n}^{+}u_{0}
  = \partial_{n}^{-}w_{0}-\partial_{n}^{+}w_{0}\,.
$$
Representing $u_{0}=\SS_{D}\phi$ and $w_{0}=\SS_{D}\psi$ and using the jump relations, we find the equivalent boundary integral equation
$$
  (\epsilon_{r}+1)\tfrac12\phi + (\epsilon_{r}-1)K'_{D}\phi = \psi\,. 
$$
We have thus shown that the operator 
$(\epsilon_{r}+1)\tfrac12\Id + (\epsilon_{r}-1)K'_{D}$ is Fredholm in $H^{\frac12}(\Gamma)$ if and only if $\Div\epsilon\grad$ is Fredholm from $H^{1}_{0}(\Omega)$ to its dual space. 

This proves the ``Dirichlet'' version of the lemma. In order to show the ``Neumann'' version, one can repeat the same steps, replacing $H^{1}_{0}(\Omega)$ by $H^{1}_{*}(\Omega)$ and the Green function $G_{D}$ by the Neumann function $G_{N}$, that is the fundamental solution satisfying homogeneous Neumann boundary conditions on $\partial\Omega$. One obtains a boundary integral equation involving a ``Neumann'' version $K'_{N}$ of $K'$, which will also differ by a compact operator from $K'$ and thus have the same essential spectrum as either $K'_{D}$ or $K'$.
\end{proof}

Combining Lemmas \ref{L:symeps} and \ref{L:Var-BIE}, it is now easy to finish the proof of Theorem~\ref{T:sym}. We find that for any $\epsilon_{r}\in\C\setminus\{0\}$, the operator
$$
  (\epsilon_{r}+1)\tfrac12\Id + (\epsilon_{r}-1)K' 
$$
is Fredholm in $H^{-\frac12}(\Gamma)$ if and only if the operator
$$
  (\tfrac1{\epsilon_{r}}+1)\tfrac12\Id + (\tfrac1{\epsilon_{r}}-1)K' 
$$
is Fredholm. Setting $\epsilon_{r}=\frac{2\lambda+1}{2\lambda-1}$, we have proved that for any 
$\lambda\ne\pm\frac12$, 
$\lambda\Id +K'$ is Fredholm if and only if $\lambda\Id -K'$ is Fredholm, which concludes the proof of the theorem.




\end{document}